\documentclass[10pt, psamsfonts]{amsart}

\usepackage[a4paper,margin=4cm]{geometry}
\usepackage{amsfonts,amssymb,amscd}
\usepackage{mathpazo}
\usepackage{hyperref} % usar solo para pdf, con los dvi puede dar problemas

\newtheorem{theorem}{Theorem}[section]
\newtheorem{lemma}[theorem]{Lemma}
\newtheorem{corollary}[theorem]{Corollary}
\newtheorem{proposition}[theorem]{Proposition}

\newtheorem{definition}{Definition}%\renewcommand{\thedefinition}{} % con esa orden, las definiciones no se numeraran

\theoremstyle{remark} % todo lo que vaya desde aqui para abajo, por ejemplo Conjecture, no aparecera en negrita

\newtheorem{remark}[theorem]{Remark}

\renewcommand{\leq}{\leqslant}
\renewcommand{\geq}{\geqslant}

\newcommand{\ptl}{\partial}

\newcommand{\ga}{\gamma}

\newcommand{\sg}{\sigma}
\newcommand{\escpr}[1]{\left< #1\right>}

\newcommand{\rr}{\mathbb{R}}
\newcommand{\sph}{\mathbb{S}}
\newcommand{\nn}{\mathbb{N}}


\begin{document}

\title[Morse index of CMC tori of revolution in $\sph^3$]{A new bound on
the Morse index of constant mean curvature tori of revolution in $\sph^3$}

\author[A. Ca\~nete]{Antonio Ca\~nete}
\address{Departamento de Matem\'atica Aplicada I \\ Campus de Reina Mercedes \\ Universidad de Sevilla \\ E-41012\  Sevilla (Espa\~na)}
\email{antonioc@us.es}

\bibliographystyle{plain} % estoy usando bibtex para las referencias; si aquí ponemos amsplain, en la bibliografia, los numeros de cada referencia aparecen sin corchetes. Poniendo plain, salen con corchetes.
% \nocite{*}

\thanks{This work was partially supported by MEC research project
MTM2007-61919}
%\date{\today}

\begin{abstract}
In this work we give a new lower bound on the Morse index for constant mean curvature tori of
revolution immersed in the three-sphere $\sph^3$,
by computing some explicit negative eigenvalues for the corresponding Jacobi operator.
%, improving some previous known bounds.
\end{abstract}

\maketitle

\thispagestyle{empty}
\section*{Introduction}
\label{sec:introduction}

Given a closed surface $M$ with \emph{constant mean curvature} (CMC),
immersed in a three-dimensional manifold $N$, it is well known that
$M$ is a critical point of the area functional, when we consider variations
preserving the enclosed volume \cite{bcs}, \cite{bb}. For this kind of surfaces,
we can discuss the \emph{stability} by studying the second variation of the area,
which can be expressed by means of an useful and classical functional operator $L$;
more precisely, denoting by $f\in\mathcal{C}^{\infty}(M,\rr)$ to the normal component of
the vector field associated to a variation of $M$,
the second variation formula will be given by
\begin{equation}
\label{eq:second}
-\int_M f\,Lf\,da,
\end{equation}
where $L=\Delta+\overline{R}+|\sg|^2$, $\Delta$ is the Laplacian operator
of $M$, $\overline{R}$ is the Ricci curvature of $N$, and $\sg$ is the second
fundamental form of $M$.
Such operator $L:\mathcal{C}^{\infty}(M,\rr)\to\mathcal{C}^{\infty}(M,\rr)$
is usually called the \emph{Jacobi} (or stability) \emph{operator}.

For CMC surfaces, $M$ is said to be \emph{stable}
if and only if the above expression \eqref{eq:second}
is greater than or equal to zero, for any volume-preserving variation \cite[\S~2]{bcs}.
In this setting, a usual way of measuring the \emph{instability} of a surface $M$
is given by the \emph{Morse index} of $M$ (see for instance \cite{abp}, \cite{rs}), which is defined as
the number of negative eigenvalues, counted with their multiplicities, of the Jacobi operator $L$
(throughout this paper, a value $\lambda\in\rr$ will be an eigenvalue of $L$ if there exists a function
$f_\lambda\in\mathcal{C}^{\infty}(M,\rr)$ such that $Lf_\lambda+\lambda f_\lambda=0$).
% In fact (or more precisely), we have that if a CMC surface $M$ is stable (from the CMC point of view), then the Morse index of $M$ is always less or equal to one (\cite{bb}). In the particular case of minimal surfaces, stability is usually considered without the mean-value condition (that is, without the volume-preserving restriction on variations), and so zero index and stability are equivalent.

This intrinsic relation with stability (see \cite{lr} for further details) has stimulated
the study of the Morse index of CMC surfaces in several works.
The main approaches have focused on the minimal case, that is, surfaces with zero mean curvature
(see \cite{montielros}, \cite{ek}, \cite{naya}, \cite{tysk}, \cite{choe}).
%,where zero index is equivalent to the stability of the surface.
In the euclidean space $\rr^3$, for instance, planes have zero index \cite{peng},
meanwhile
catenoids and Enneper's surfaces have index one \cite{fc}, \cite{lr}.
In the sphere $\sph^3$, an interesting result from J. Simons \cite{simons} states that the index
of any compact minimal surface $M$ is always greater than or equal to one,
%(and therefore $M$ cannot be stable)
with equality if and only if $M$ is a totally geodesic $2$-sphere
(in fact, such a result was stated in general dimension).
Later on, F. Urbano \cite{urbano} proved that any compact minimal surface (not totally geodesic)
in $\sph^3$ has index greater than or equal to five, with equality uniquely for the Clifford torus.
The analogous result in the $n$-dimensional case has been partially demonstrated in \cite{elsoufi}, \cite{guada}, \cite{perdomo}. A nice review of these results can be found in \cite{alias}.

However, in the family of nonminimal CMC surfaces, the Morse index has been less discussed in literature,
and the main results refer to bounds for the index.
In $\rr^3$, apart from the spheres (which have index one), only lower bounds for tori of revolution \cite{rossman-tori}, and upper and lower bounds for
the Wente tori \cite{wente} are known.
Moreover, for CMC immersions of revolution in $\sph^3$, W. Rossman and N. Sultana \cite{rs}
have recently computed the index of flat tori of revolution (in terms of the mean curvature),
and have also found a lower bound for non-flat tori,
giving as well numerical methods for explicit calculations \cite{rs2}. It follows from
their work that the index of these tori is at least five, in any case.

In this work, we focus on the index of CMC tori of revolution immersed in $\sph^3$.
By using an approach different from the one used in \cite{rs}, we shall explicitly find
some negative eigenvalues for the Jacobi operator of these surfaces,
obtaining new bounds for the Morse index.
Our technique implies a suitable arrangement of the metric of the surface.
This will lead us to have a nice expression of the Jacobi operator,
which allows to determine some specific eigenfunctions and eigenvalues.
We remark that these eigenvalues will depend on some geometric quantities associated to the surface
(as the energy and the mean curvature), and we will discuss analytically their sign,
improving in most of the cases the previous known bounds for the Morse index (see Theorem~\ref{th:main}).
For instance, we obtain that when the mean curvature of the torus is less than $-1$ or greater than $3/2$,
the Morse index is, at least, eight.
%XXXXXXXXXX Detallar algo mas los resultados obtenidos.

% Indice Finito en $\rr^3$: Fischer-Colbrie \cite{fc} proved that a complete minimal surface in $\rr^3$ has finite Morse index if and only if it has finite total curvature (see also \cite{gulliver}). Later on, L\'opez and Ros \cite{lr} improved this result, stating that a complete CMC surface has finite index if and only if it compact (or minimal with finite total curvature).

We point out that an interesting question, in this CMC setting, is finding a similar result to that of F. Urbano,
who characterized Clifford tori in $\sph^3$ by means of a minimum value of the Morse index \cite{urbano}.
Some partial progresses have been made in this direction \cite{abp}.

\emph{Acknowledgments.}  The author would like to thank Manuel Ritor\'e
for his continuous support and kind help during the elaboration of these notes.

\section{Preliminaries}
Consider a torus of revolution $M$ with constant mean curvature $H$,
immersed in the $3$-dimensional unit sphere $\sph^3\subset\rr^4$.
Let us fix a geodesic curve $\ell$ in $\sph^3$, given by
\[
\ell=\{(\cos(t),\sin(t),0,0),\ t\in\rr\}.
\]
Then, the torus can be seen as generated by the rotation about $\ell$ of a given curve
parameterized by arc-length
$$\ga:[0,t_0]\longrightarrow\sph^3$$
defined by
\[
\ga(t)=(\ga_1(t),\ga_2(t),\ga_3(t),0), \quad t\in[0,t_0].
\]
We can further assume that $\ga_3(t)>0$ for all $t$.

\begin{remark}
The description of the generating curve of a CMC torus of revolution
can be found in~\cite[Th.~3]{hsiang} (see also~\cite[\S~1]{rs}),
being unduloidal or nodoidal.
\end{remark}

%\begin{remark}
%A torus of revolution with constant mean curvature immersed in $\sph^3$ can be of different
%types, depending on its generating curve.
%Such a periodic curve may be unduloidal (thus yielding an embedded torus) or nodoidal,
%and will attain the maximum and the minimum distance to the axis of revolution
%in different points called bulges and necks (see~\cite[Th.~3]{hsiang},~\cite[\S~1]{rs}).
%In this setting, the constant value of the first integral \eqref{eq:first}
%will determine the character of the torus.
%%In particular, the number of bulges and necks will coincide with the number of nodal regions provided by the Killing vector field of rotation in the direction of $\ell$ (that is, rotation about an orthogonal axis to $\ell$).
%\end{remark}

Hence, an immersion $\psi:M\to\sph^3$ of the torus into $\sph^3$ will be given by
\[
\psi(t,\theta)=(\ga_1(t),\ga_2(t),\ga_3(t)\cos\theta,\ga_3(t)\sin\theta),\quad t\in [0,t_0],\ \theta\in[0,2\pi].
\]
Then we have that the tangent space to $\psi(t,\theta)$ is generated by the vectors
\begin{align}
\label{eq:tangentes}
\partial_t&=(\ga_1'(t),\ga_2'(t),\ga_3'(t)\cos\theta,\ga_3'(t)\sin\theta),
\\
\partial_\theta&=(0,0,-\ga_3(t)\sin\theta,\ga_3(t)\cos\theta), \notag
\end{align}
with $|\partial_t|=1$,$|\ptl_\theta|=\ga_3(t)$.
Moreover, it is straightforward checking that the unit normal vector to $\psi(t,\theta)$
will be given by
\begin{equation}
\label{eq:normal}
N(t,\theta)=
\left( \begin{array}{c}
\ga_2(t)\ga_3'(t)-\ga_2'(t)\ga_3(t)
\\
\ga_3(t)\ga_1'(t)-\ga_3'(t)\ga_1(t)
\\
(\ga_1(t)\ga_2'(t)-\ga_1'(t)\ga_2(t))\,\cos\theta
\\
(\ga_1(t)\ga_2'(t)-\ga_1'(t)\ga_2(t))\,\sin\theta
\end{array} \right)
=
\left( \begin{array}{c}
a(t)
\\
b(t)
\\
c(t)\,\cos\theta
\\
c(t)\,\sin\theta
\end{array} \right).
\end{equation}

Since $M$ is rotationally symmetric, we can consider in $M$ a metric of type
$$ds^2=dt^2+f(t)^2\,d\theta^2,$$
for certain positive function $f:M\to\rr$.
We are interested in computing the Hopf differential of our immersion $\psi$,
and so we shall need conformal (isothermal) coordinates in $M$.
%We need conformal (isothermal) coordinates in $M$ in order to compute the Hopf differential of our immersion.
In order to have this, we use the following change of coordinates:
consider a new coordinate $t'$ defined by $t=G(t')$,
where $G$ is a diffeomorphism of $\rr$ determined by $G'(t')=f(G(t'))$ and with initial condition $G(0)=0$.
Defining now $$g(t',\theta)=(G(t'),\theta),$$
it is not difficult to check that the new immersion of $M$ given by
$$(\psi\circ g)(t',\theta)=\psi((G(t'),\theta))$$
is conformal, with
$$ds^2=f(G(t'))^2\,\big((dt')^2+d\theta^2\big).$$
Observe that we have properly replaced coordinates $(t,\theta)$ with $(t',\theta)$.
It also follows from above that
\begin{equation}
\label{eq:igualdad}
|\ptl_\theta|=f(t)=\ga_3(t), \text{ for all } t.
\end{equation}

%It is not difficult to check that considering a new coordinate $t'$ defined by $t=G(t')$,
%where $G$ is a diffeomorphism of $\rr$ determined by $G'(t')=f(G(t'))$ and with initial condition $G(0)=0$,
%% tenemos libertad sobre la condición inicial; tomamos esa porque facilitará los cálculos mas adelante, por ejemplo los referentes a $w(0)$.
%by defining $$g(t',\theta)=(G(t'),\theta),$$ we have a new immersion of $M$ given by
%$$(\psi\circ g)(t',\theta)=\psi((G(t'),\theta)),$$
%such that $(t',\theta)$ are now conformal coordinates in $M$ with
%$$ds^2=f(G(t'))^2\,\big((dt')^2+d\theta^2\big).$$
%From above we observe that $|\ptl_\theta|=f(t)=\ga_3(t)$, for all $t$.

Now, by considering the flat metric $ds_0^2$ associated to the Hopf differential
of the immersion $\psi\circ g$, it follows that
\begin{equation}
\label{eq:conforme}
ds^2=\exp(2w)b^{-2}\,ds_0^2,
\end{equation}
where $b^2=4(1+H^2)$, and $w:M\to\rr$ is a smooth function defined on the torus
(see \cite[\S~1]{ritros}). % Nos interesa que aparezcan la metrica ds_0^2 y la funcion w, para tener mas adelante que cosh y senh son funciones propias

\begin{remark}
We point out that
\begin{equation}
\label{eq:beta}
ds_0^2=\beta\,((dt')^2+d\theta^2),
\end{equation}
for certain $\beta\in\rr$, due to the flatness of the metric $ds_0^2$.
From the above expressions \eqref{eq:conforme}, \eqref{eq:beta} we get
\begin{equation}
\label{eq:betapositiva}
\beta=\exp(-2w(t'))\,b^2\,f(t)^2>0,
\end{equation}
or equivalently
\[
f(t)=\exp(w(t'))\,b^{-1}\,\beta^{1/2}.
\]
\end{remark}

\subsection{On the value of the constant $\beta$}
The value of the constant $\beta$ can be computed in terms of the
principal curvatures $k_t$, $k_\theta$ of the immersion $\psi$, taking into account the construction
of the Hopf differential. More precisely, denoting
by $\sigma'$ the second fundamental form of the conformal immersion $\psi\circ g$, we have (see~\cite[\S~1]{ritros})
\begin{equation}
\label{eq:beta1}
\beta=\sigma'(\ptl_{t'},\ptl_{t'})-\sigma'(\ptl_{\theta},\ptl_{\theta})=
f(t)^2\,(k_t - k_{\theta}),
\end{equation}
%where $\ptl_t=(\ga_1'(t),\ga_2'(t),\ga_3'(t)\cos(\theta),\ga_3'(t)\sin(\theta))$, $\ptl_\theta=(0,0,-\sin(\theta),\cos(\theta))$ are the unit tangent vectors of $M$.
where, taking into account \eqref{eq:tangentes}, \eqref{eq:normal}, the principal curvatures are given by
\begin{align}
\label{k1}
k_\theta&=\sigma(\ptl_\theta,\ptl_\theta)=\frac{\ga_1'(t)\ga_2(t)-\ga_1(t)\ga_2'(t)}{\ga_3(t)}=\frac{-c(t)}{\ga_3(t)},
%\]
%and
%\[
\\
k_t&=\sigma(\ptl_t,\ptl_t)=-
\left|\begin{array}{ccc}
\ga_1(t) & \ga_2(t) & \ga_3(t)
\\
\ga_1'(t) & \ga_2'(t) & \ga_3'(t)
\\
\ga_1''(t) & \ga_2''(t) & \ga_3''(t)
\end{array}\right|.  \notag
\end{align}
In order to calculate explicitly the value of $\beta$ (which we shall need later),
we will take into account some particular parametrization of the points in $\sph^3$
appearing in \cite[\S~1]{hsiang}. This will provide, in particular,
new useful expressions for the principal curvatures $k_t$, $k_\theta$.
%Fix a base point in $\ell$. % and let $x\in\rr^+$ be the arc length on $\ell$.

For any point $p\in M\subset\sph^3$, there exists a point $q\in\ell$ such that
$\overline{pq}$ is a geodesic arc whose length is equal to the distance between $p$ and $\ell$.
Fixing a base point in $\ell$ and considering the arc length on $\ell$ with respect to the base point,
we can assign to $p$ new coordinates $(x,y)$, where $x$ is the coordinate of $q$ in $\ell$,
and $y$ is the length of the geodesic $\overline{pq}$.
This procedure allows to express the generating curve $\ga\subset\sph^3$ in terms of coordinates $x$, $y$.
Straightforward computations yields
\begin{align}
\label{gamma}
\ga_1(t)&=\cos(x(t))\cos(y(t)),
\\
\ga_2(t)&=\sin(x(t))\cos(y(t)), \notag
\\
\ga_3(t)&=\sin(y(t)).    \notag
\end{align}
Moreover, since the mean curvature $H$ is constant, the following relations hold (\cite[eq.~(4)]{hsiang}):
\begin{equation}
\label{eq:relations}
x'(t)=\frac{\sin(\alpha(t))}{\cos(y(t))}, \quad y'(t)=\cos(\alpha(t)),
\end{equation}
where $\alpha(t)$ is the angle between the tangent vector to $\gamma(t)$ and the vertical direction.

Then, from \eqref{k1}, \eqref{gamma} we get that
\begin{align}
\label{kt}
&k_\theta=-\sin(\alpha(t))\,\cot(y(t)),
\\
&k_t=\sin(\alpha(t))\,\tan(y(t))-\alpha'(t). \notag
\end{align}
We recall that
\begin{equation}
\label{H}
2H=k_t+k_\theta=-\sin(\alpha(t))\,\cot(y(t))+\sin(\alpha(t))\,\tan(y(t))-\alpha'(t).
\end{equation}

On the other hand, it is easy to check that
\begin{equation}
\label{energia}
E=\sin(y(t))\,\bigg(\cos(y(t)\,\sin\alpha(t)+H\,\sin(y(t))\bigg)
\end{equation}
is constant (just compute the derivative with respect to $t$, using the above relations~\eqref{eq:relations}).

Finally, taking into account previous expressions
\eqref{eq:beta1},~\eqref{kt},~\eqref{H} and~\eqref{energia}, %~\eqref{eq:relations},
it follows that
%\[
%k_\theta=-\sin\alpha(t)\,\cot(y(t)),\quad\quad k_t=\sin\alpha(t)\,\tan(y(t))-\alpha'(t)
%\]
%and
\begin{equation}
\label{eq:beta2}
\beta=f(t)^2\,(k_t - k_{\theta})=\sin^2(y(t))\,\bigg(2\,\sin\alpha(t)\,\cot(y(t))+2\,H\bigg)=2\,E.
\end{equation}

\begin{remark}
In fact, the above constant $E$ is a first integral associated to the system of equations \eqref{eq:relations}
of the generating curve $\gamma$ (see \cite[Th.~1]{hsiang}).
\end{remark}

\begin{remark}
In the previous computations of $k_t$, $k_\theta$, we have considered the normal vector $N(t,\theta)$
defined by \eqref{eq:normal}. %(by means of the vectorial product of $\ptl_t$, $\ptl_\theta/|\ptl_\theta|$ and $\psi(t,\theta)$).
If we consider that normal vector with opposite sign, the values of $k_t$, $k_\theta$ will change the
sign too. %, and will coincide exactly with the computations done in \cite{hsiang}.
The appropriate choice of $N(t,\theta)$ will be determined by the positivity of $\beta$, in view of \eqref{eq:beta1}.
%As the constant $\beta$ must be positive, we may need considering the opposite normal vector, in view of \eqref{eq:beta1}.
% si con mi eleccion del normal, beta sale negativo, cambio de normal y beta pasará a ser positivo.
\end{remark}

\begin{remark}
It follows from \eqref{eq:igualdad}, \eqref{gamma} that
\begin{equation}
\label{eq:f}
f(t)=\ga_3(t)=\sin(y(t)),\ t\in[0,t_0].
\end{equation}
\end{remark}

\subsection{On the function $w$}

The function $w$ arising in~\eqref{eq:conforme} establishes the relation between the metric $ds^2$ in $M$
and the flat metric $ds_0^2$ associated to the Hopf differential of the immersion.
It only depends on the variable $t$ since $M$ is rotationally symmetric,
and it satisfies the sinh-Gordon equation for the laplacian associated to $ds_0^2$ (\cite[eq.~(2)]{ritros});
equivalently, by using~\eqref{eq:beta} we have
\begin{equation}
\label{eq:sinhG}
\left\{\begin{tabular}{l}
    $w''+\beta\sinh(w)\cosh(w)=0$,
    \\[2mm]
    $w(0)=c, w'(0)=0,$
    \end{tabular}\right.
\end{equation}
where $c\in\rr$, and the derivatives are taken with respect to the standard flat metric $(dt')^2+d\theta^2$.

%Es necesario tener coordenadas conformes o isotermas para trabajar con la diferencial de Hopf.

% la metrica ds_0^2 es llana, pero vendra dada por cierta constante beta de la siguiente manera: $ds_0^2=\beta((dt')^2+d\theta^2)$. Es decir, es \beta-proporcional a la métrica llana euclidea

% XXXXXXXX Falta determinar el valor de la constante $\beta$, que aparecerá mas tarde.

% El siguiente remark (que simplemente indica de forma breve el valor de la constante $w(0)=c$) habra que revisarlo, ya que la primera integral es ahora distinta (tener cuidado con la metrica que se considera en cada caso). Se podrá determinar parametrizando la metrica de la esfera como $dr^2+cos^2(r)d\theta^2$, tal y como hace Hsiang.

\begin{remark}
\label{re:c}
The above value $c\in\rr$ in \eqref{eq:sinhG} is related with the length of the parallel $\sph^1\times\{0\}\subset M$,
since that length is equal to $2\pi\,f(0)=2\pi\,\exp(c)\,b^{-1}\,\beta^{1/2}$.
On the other hand, by using \cite[eq.~(1)]{hsiang} and \eqref{eq:f},
we have that $L(\sph^3\times\{0\})=2\pi\,\cos(y(0))=2\pi\sqrt{1-f(0)^2}$.
Thus, the value $c$ equals
\[
\log\Bigg(\frac{b\ \sqrt{1-\ga_3^2(0)}}{\beta^{1/2}}\Bigg)=
\log\Bigg(\frac{b\ \sqrt{1-f(0)^2}}{\beta^{1/2}}\Bigg).
\]
% Here we are also using the relation between $f$ and $w$, and that $G(0)=0$.
\end{remark}

%\begin{remark}
%Above equality~\eqref{eq:conforme} has been obtained
%by means of the Hopf differential of the immersion $\psi$ and its associated flat metric,
%taking into account the conformal coordinates $(t',\,\theta)$ considered in $M$ (see \cite{ritros}).
%\end{remark}

In this setting, we recall that the Gauss curvature $K$ of $M$, with respect to the metric $ds^2$,
only depends on the $t$-coordinate and is given by
\begin{equation}
\label{eq:K}
K=(b^2/4)(1-\exp(-4w)).
\end{equation}

\begin{remark}
By integrating equality \eqref{eq:sinhG}, multiplied by $2\,w'$, %and taking into account \eqref{eq:beta},
we obtain
\begin{equation}
\label{eq:first}
(w')^2+\beta\cosh^2(w)=\beta\cosh^2(c),
\end{equation}
which is a first integral of equation \eqref{eq:sinhG},
where the derivative is with respect to the flat metric $(dt')^2+d\theta^2$.
\end{remark}

\subsection{Index form and Jacobi operator}
Recall that $N$ denotes the normal vector field of $M$, and $K$ is the Gauss curvature.
%Let us denote by $N$ the normal vector field of $M$, and by $K$ the Gauss curvature.
Then, it is well known that the second variation formula of the area,
for variations preserving the volume enclosed by $M$, is given in general by \cite[Prop. 2.5]{bcs}
\begin{align}
\label{eq:indexform}
I(f,f)&=-\int_M \big(f\Delta f + (\overline{R}+|\sg|^2)f^2\big)\,da
\\
&=-\int_M \big(f\Delta f +(4+4 H^2- 2K)f^2\big)\,da, \notag
\end{align}
where $\Delta$ is the laplacian operator associated to the metric $ds^2$,
$f$ is the normal component of the vector field associated to the variation,
$\overline{R}=2\,Ric(N)$ is the Ricci curvature of the ambient space $\sph^3$,
and $|\sg|^2$ is the square of the norm of the second fundamental form $\sg$.
%As $\overline{R}=4$ and $|\sg|^2=4\,H^2-2\,K$ in this case,
We shall refer to the quadratic form defined by \eqref{eq:indexform} as the
\emph{index form} of $M$.
The associated \emph{Jacobi operator} is thus given by
\[
L f=\Delta f + (4+4 H^2- 2K) f,
\]
for any $C^{\infty}(M)$ function $f$, and so
\[
I(f,f)=-\int_M f\,Lf\,da.
\]

\noindent Taking into account expression %\eqref{eq:metricas} and
\eqref{eq:K} we have that
\begin{align}
\label{cambio}
4+4H^2-2K&=b^2-2K= (b^2/2)\,(1+\exp(-4w)) \notag
\\
&=b^2\,\exp(-2w)\,\cosh(2w) \notag
\\
&=b^2\,\exp(-2w)\,(\cosh^2(w)+\sinh^2(w)).
\end{align}
In view of expression \eqref{cambio}, we now define a new operator $L_0$ by
$$L_0 f=\exp(2w)b^{-2} Lf,\quad f\in C^{\infty}(M).$$
Then, it is clear, in view of \eqref{eq:conforme} and \eqref{eq:beta}, that
\begin{align}
\label{eq:L0}
L_0 f=&\Delta_0f+(\cosh^2(w)+\sinh^2(w))f
\\
=&\frac{1}{\beta}\bigg(\frac{\ptl^2f}{\ptl t^2}+\frac{\ptl^2f}{\ptl\theta^2}\bigg) + (\cosh^2(w)+\sinh^2(w))f, \notag
\end{align}
where $\Delta_0$ represents the laplacian of $ds_0^2$.
An important fact that we will use later is that both operators $L$, $L_0$
only differ in a positive scalar factor.
%, so that they have the same index and nullity

%\[
%I(f,f)=-\int_M\big(f\Delta_0 f+(\cosh^2(w)+\sinh^2(w))f^2\big)\,da_0,
%\]
%with respect to the flat metric $ds_0^2$, with Jacobi operator given by
%\begin{equation}
%\label{eq:jacobi}
%Lf=-\Delta_0 f-(\cosh^2(w)+\sinh^2(w))f,\quad f\in C^{\infty}(M).
%\end{equation}

\subsection{Morse index of CMC surfaces}
The \emph{Morse index} of a closed constant mean curvature (CMC) surface $\mathcal{S}$
is defined by means of the Jacobi operator and,
as indicated in \cite{rs}, provides a degree of the instability of $\mathcal{S}$ with respect to the area.
We first recall the following definition.

\begin{definition}
Given a function $f:M\to\rr$, we shall say that $f$ is an eigenfunction of the Jacobi operator $L$,
with associated eigenvalue $\lambda\in\rr$, if %it satisfies that
\[
Lf+\lambda f=0.
\]
\end{definition}

\noindent It is known that the set of eigenvalues $\{\lambda_i\}_{i\in\nn}$ of the Jacobi operator $L$
consists of an increasing sequence, diverging to $+\infty$,
and that the first eigenvalue $\lambda_1$ has multiplicity one (see \cite{chavel} for further details).
We can now define the Morse index of a CMC surface.
%and the nullity
\begin{definition}
Given a closed CMC surface $\mathcal{S}$, the Morse index Ind($\mathcal{S}$) is the number of negative
eigenvalues of the Jacobi operator $L$, each one counted with its multiplicity.
%The nullity of $\mathcal{S}$ is the multiplicity of the zero eigenvalue of $L$.
\end{definition}

Our purpose is giving a lower bound for the Morse index for CMC tori of re\-volution immersed in $\sph^3$.
To do that, we will focus on the operator $L_0$, since due to its definition,
it will have the same number of negative eigenvalues that the Jacobi operator $L$.
%In addition, it is clear that both operators have the same nullity.

\section{Explicit computation of some eigenvalues of $L_0$}
In this Section we shall compute directly some (negative) eigenvalues of the ope\-rator $L_0$,
taking into account its expression \eqref{eq:L0}, by using certain functions on $M$ with independent variables.
We first define the operator $L_t$, on the set of functions of real variable, as
\[
L_tf=\frac{1}{\beta}\,f''(t)+(\cosh^2(w)+\sinh^2(w))f(t), \quad f:\rr\to\rr.
\]
It is clear that, for functions defined on $M$, we have that
$L_0=\displaystyle{\frac{1}{\beta}\ \frac{\partial^2}{\partial\theta^2} + L_t}$.
We begin with the following key result.

\begin{lemma}
\label{le:lambdamu}
Let $u=u(t)$ be an eigenfunction of $L_t$, with associated eigenvalue $\lambda\in\rr$, and
let $v=v(\theta)$ be an eigenfunction of the laplacian $\frac{\ptl^2}{\ptl\theta^2}$, with associated eigenvalue $\mu\in\rr$.
Then, the function $f:M\to\rr$ given by $f(t,\theta)=u(t)v(\theta)$ is an eigenfunction of $L_0$,
with associated eigenvalue $\lambda+\mu/\beta$.
\end{lemma}

\begin{proof}
Applying the operator $L_0$ to $f$, we have that
\[
L_0 f=\frac{1}{\beta}\ \frac{\ptl^2 v}{\ptl\theta^2}\,u+v\,L_t u=-(\lambda+\frac{\mu}{\beta})u\,v=-(\lambda+\frac{\mu}{\beta})f,
\]
and so the result follows.
\end{proof}

We now proceed to find convenient functions for applying Lemma~\ref{le:lambdamu}.
It is clear that $v(\theta)$ can be taken equal to a constant, or equal to $\cos(\theta)$, $\sin(\theta)$,
which are eigenfunctions of the laplacian (with $\mu=0$ and $\mu=1$ as eigenvalues, respectively).
The next result gives some eigenfunctions for the operator $L_t$.

\begin{proposition}
\label{prop:hiperbolicos}
The functions $u_1,u_2:[0,t_0]\to\rr$ defined by
\[
u_1(t)=\cosh(w(t)),\quad u_2(t)=\sinh(w(t)),
\]
are independent eigenfunctions of $L_t$, with associated eigenvalues
$$\lambda_1=-\cosh^2(c),\quad \lambda_2=1-\cosh^2(c),$$ respectively.
\end{proposition}

\begin{proof}
The proof is straightforward, taking into account \eqref{eq:sinhG} and \eqref{eq:first}.
\end{proof}

\begin{remark}
We point out that when $w$ is identically zero (which corresponds to flat metric $ds^2$),
previous Proposition~\ref{prop:hiperbolicos} only shows that
constant functions will be eigenfunctions of $L_t$, with eigenvalue $\lambda=-1$.
\end{remark}

Our idea consists of using functions with independent variables.
By combining the functions $u_1(t)$, $u_2(t)$ from Proposition~\ref{prop:hiperbolicos}
with a constant function, with $\sin(\theta)$, or with $\cos(\theta)$,
we can apply Lemma~\ref{le:lambdamu} to obtain some eigenvalues for the operator $L_0$.

\begin{theorem}
\label{th:six}
Let $M$ be a CMC torus of revolution immersed in $\sph^3$ and $L_0$ the operator defined previously.
Then,
\begin{itemize}
%\item[i)] the first eigenvalue of $L_0$ is $\lambda_1=-\cosh^2(c)$, with eigenfunction  $f_1(t,\theta)=\cosh(w(t))$,
\item[i)] the first eigenfunction of $L_0$ is $f_1(t,\theta)=\cosh(w(t))$, with associated eigenvalue $\lambda_1=-\cosh^2(c)<0$,
\item[ii)] $f_2(t,\theta)=\sinh(w(t))$ is an eigenfunction of $L_0$, with associated eigenvalue $1-\cosh^2(c)<0$,
\item[iii)] $f_3(t,\theta)=\cosh(w(t))\sin(\theta)$, $\overline{f_3}(t,\theta)=\cosh(w(t))\cos(\theta)$ are two eigenfunctions
of $L_0$, with associated eigenvalue $-\cosh^2(c)+1/\beta$,
\item[iv)] $f_4(t,\theta)=\sinh(w(t))\sin(\theta)$, $\overline{f_4}(t,\theta)=\sinh(w(t))\cos(\theta)$ are two eigenfunctions
of $L_0$, with associated eigenvalue $1-\cosh^2(c)+1/\beta$.
\end{itemize}
Furthermore, these six eigenfunctions are independent, and Ind($M$)$\geq 2$.
\end{theorem}

\begin{proof}
Just apply Lemma~\ref{le:lambdamu} with the corresponding functions in order to
obtain the six independent eigenfunctions.
Moreover, $f_1$ is the first eigenfunction of $L_0$ since it does not vanish.
Finally, $\lambda_1=-\cosh^2(c)$ is always negative, and since $c\neq0$
(otherwise, \eqref{eq:sinhG} yields $w=0$ and so $f$ is constant), we have that $1-\cosh^2(c)$
is also a negative eigenvalue for $L_0$.
\end{proof}

Another negative eigenvalue of $L_0$ can be obtained by following some ideas from \cite{rs}.
For a given geodesic curve $\ell'$ in $\sph^3$, orthogonal to $\ell$,
we can consider the Killing vector field $K$ associated to the rotations about $\ell'$ in $\sph^3$.
Then, the normal component $f=\escpr{K,N}$ of $K$ satisfies $L_0(f)=0$ (that is, $f$ is a Jacobi function for $L_0$),
and can be expressed as $f(t,\theta)=u(t)\cos\theta$ (see \cite[Lemma~4.1]{rs}).
In this situation, it is easy to check that $u$ is an eigenfunction of $L_0$, with associated eigenvalue $-1/\beta$
(which is negative since $\beta>0$).
Therefore, as in \cite[Lemma~4.2]{rs}, by taking the two geodesic curves orthogonal to $\ell$,
we get two independent eigenfunctions of $L_0$ (depending only on variable $t$) with eigenvalue $-1/\beta$.
Straightforward computations show that these two eigenfunctions are given by
\begin{align}
\label{eq:uu}
u(t,\theta)&=-\ga_3(t)\,b(t)+\ga_2(t)\,c(t)=-\ga_1'(t),
\\
\overline{u}(t,\theta)&=-\ga_3(t)\,a(t)+\ga_1(t)\,c(t)=\ga_2'(t), \notag
\end{align}
where $a(t)$, $b(t)$, $c(t)$ are the real functions provided by~\eqref{eq:normal}.
Lemma~\ref{le:uu} summarizes these properties.

%Two more eigenfunctions of $L_0$ with a negative eigenvalue can be obtained following some ideas from \cite{rs}
%(see \cite[Lemmata 4.1 and 4.2]{rs}).
%Given a geodesic curve $\ell'$ in $\sph^3$, orthogonal to $\ell$,
%we can consider the Killing vector field $K$ associated to the rotations about $\ell'$ in $\sph^3$.
%Then, we have that the normal component $f=\escpr{K,N}$ of $K$ satisfies $L_0(f)=0$ (that is, $f$ is a Jacobi function for $L_0$),
%and can be expressed as $f(t,\theta)=u(t)\,\cos\theta$.
%In this situation, it is easy to check that $u$ is an eigenfunction of $L_0$, with associated eigenvalue $-1/\beta$,
%which is negative since $\beta>0$.
%Therefore, as in \cite[Lemma~4.2]{rs}, by considering the two geodesic curves orthogonal to $\ell$,
%we shall obtain two independent eigenfunctions of $L_0$ with eigenvalue $-1/\beta$.
%Straightforward computations show that these two functions are given by
%\begin{align*}
%u(t)&=-\ga_3(t)\,b(t)+\ga_2(t)\,c(t)=-\ga_1'(t),
%\\
%\overline{u}(t)&=-\ga_3(t)\,a(t)+\ga_1(t)\,c(t)=\ga_2'(t),
%\end{align*}
%where $a(t)$, $b(t)$, $c(t)$ are the real functions provided by~\eqref{eq:normal}.

\begin{lemma}$($\cite[Lemmata~4.1 and 4.2]{rs}$)$
\label{le:uu}
The functions $u,\,\overline{u}$ defined by \eqref{eq:uu} are two independent eigenfunctions of $L_0$ with $-1/\beta$ as
associated (negative) eigenvalue.
\end{lemma}

From above lemma we have the following interesting consequence related with Theorem~\ref{th:six}.

\begin{lemma}
\label{le:otronegativo}
The eigenvalue $-\cosh^2(c)+1/\beta$ is negative.
\end{lemma}

\begin{proof}
From Theorem~\ref{th:six} we have that $\lambda_1=-\cosh^2(c)$ is the first eigenvalue of $L_0$.
Then, since $-1/\beta$ is another eigenvalue of $L_0$, we necessarily have $-\cosh^2(c)<-1/\beta$,
and so $-\cosh^2(c)+1/\beta<0$.
\end{proof}

The Morse index is defined taking into account the multiplicities of the negative eigenvalues. %Concerning the Morse index
In this sense, we have to study carefully whether $u$ or $\overline{u}$
belong to the eigenfunctions space associated to one of the eigenvalues described in Theorem~\ref{th:six}
(if this is the case, they will not contribute to the Morse index).
For having that, a necessary condition is that $-1/\beta$ coincides with one of the eigenvalues previously obtained.
It is clear that it cannot be equal to $\lambda_1=-\cosh^2(c)$, otherwise the first eigenvalue of $L_0$
would have multiplicity greater than one, which is a contradiction (recall that $u$, $\overline{u}$ are independent).
If $-1/\beta$ coincides with $\lambda=-\cosh^2(c)+1/\beta$, as $u$, $\overline{u}$ only depend on variable $t$,
they will be independent from $f_3$, $\overline{f_3}$, and so the multiplicity of $\lambda$ will be greater than
(or equal to) four (hence contributing to the Morse index, since $\lambda<0$ from Lemma~\ref{le:otronegativo}).
The same reasoning is valid if $-1/\beta$ coincides with $\lambda'=1-\cosh^2(c)+1/\beta$
(observe that the sign of $\lambda'$ has not been discussed yet).
Finally, we have to study if $-1/\beta$ coincides with $1-\cosh^2(c)$, equivalently $1-\cosh^2(c)+1/\beta=0$.
This case of equality is treated in Subsection~\ref{subsec:degenerate}. %, a further analysis is needed.
However, if this equality does not occur, we immediately have the following result,
which establishes lower bounds for the Morse index.

\begin{theorem}
\label{th:distintos}
Let $M$ be a CMC torus of revolution immersed in $\sph^3$.
With the previous notation, assume that $1-\cosh^2(c)+1/\beta\neq0$.
Then,
\begin{itemize}
\item[i)] Ind($M$)$\geq8$, if \ $1-\cosh^2(c)+1/\beta<0$.
\item[ii)] Ind($M$)$\geq6$, if \ $1-\cosh^2(c)+1/\beta>0$.
\end{itemize}
\end{theorem}

\begin{proof}
From the hypothesis we have that the eigenfunctions shown in Theorem~\ref{th:six} and Lemma~\ref{le:uu} are
independent. An analysis of the eigenvalue $1-\cosh^2(c)+1/\beta$ yields the statement.
\end{proof}
%% La prueba de este teorema podría ser justamente el parrafo anterior.

\begin{remark}
Note that Theorem~\ref{th:distintos} establishes bounds on the Morse index of CMC tori of revolution in $\sph^3$
which improve the ones given in \cite[Th.~1.1]{rs}.
\end{remark}
% ellos llegan, en todo caso, a que el indice es siempre mayor o igual que cinco.

Observe that above Theorem~\ref{th:distintos} yields a numerical criterion
(based only on the sign of $1-\cosh^2(c)+1/\beta$, which depends on the constant values of $c$ and $\beta$)
that provides lower bounds for the Morse index.
Now, we shall express the eigenvalue $\displaystyle{1-\cosh^2(c)+1/\beta}$ in terms of the mean curvature $H$
and the value $f(0)$, in order to study its sign.

%Theorem \ref{th:six} assures that the Morse index is, at least, two.
%In order to discuss the sign of the eigenvalues $-\cosh^2(c)+1/\beta$ and $1-\cosh^2(c)+1/\beta$,
%we shall first express both of them in terms of the mean curvature $H$ and the value $f(0)$.
%
%We shall now study the sign of the eigenvalues $-\cosh^2(c)+1/\beta$ and $1-\cosh^2(c)+1/\beta$. In order to do this, we will first express both eigenvalues in terms of the mean curvature $H$ and the value $f(0)$.
% which a priori depend on the constants $\beta$ and $c$.
%
%
From Remark~\ref{re:c} we have that
\[
\cosh(c)=\frac{e^c+e^{-c}}{2}=\frac{b^2\,(1-f(0)^2)+\beta}{2\,\beta^{1/2}\,b\,\sqrt{1-f(0)^2}},
\]
and so
\begin{equation}
\label{eq:c}
\cosh^2(c)=\frac{\bigg(4(1+H^2)(1-f(0)^2)+\beta\bigg)^2}{16\,\beta\,(1+H^2)(1-f(0)^2)}.
\end{equation}

On the other hand, the constant $\beta$ can be expressed only in terms of $H$ and $f(0)$:
in fact, its value can be obtained by taking $t=0$ in \eqref{energia}, and so
\[
\beta=2\,\sin(y(0))\bigg(\cos(y(0))\,\sin(\alpha(0))+H\,\sin(y(0))\bigg).
\]
Since $f(t)=\sin(y(t))$, and taking $t=0$ such that $\alpha(0)=\pi/2$,
it follows that
\begin{equation}
\label{eq:beta4}
\beta=2\,f(0)\bigg(\sqrt{1-f(0)^2}+H\,f(0)\bigg).
\end{equation}
Using equalities \eqref{eq:c} and \eqref{eq:beta4}, the eigenvalue $1-\cosh^2(c)+1/\beta$
can be expressed only in terms of $H$ and $f(0)$.
Moreover, since $\beta>0$, it follows that $\sqrt{1-f(0)^2}+H\,f(0)$ must be positive,
and so necessarily $$f(0)<\sqrt{\frac{1}{1+H^2}},$$
if $H<0$ (in the case $H\geq0$, then $f(0)\in(0,1)$).
%% si H=0, es inmediato; si H>0, es inmediato que $\sqrt{1-f(0)^2}+H\,f(0)$ es positivo; si H<0, se opera sin problemas y se llega a esa condicion.

Above equalities allow to calculate explicitly that eigenvalue,
for each $H\in\rr$ and each considered value of $f(0)$.
Numerical computations show that both possibilities from Theorem~\ref{th:distintos} may occur for
different tori.
For instance, when $H=1$ and $f(0)=0.3$, we have that $1-\cosh^2(c)+1/\beta$ is negative,
and then we can claim that the Morse index of the corresponding torus is greater than or equal to $8$.
On the other hand, when $H=0.5$ and $f(0)=0.5$, that eigenvalue is positive, and so the Morse
index will be greater than or equal to $6$. More examples of these two situations arise
for different values of $H$ and $f(0)$.
% para H=0, ese valor propio siempre es positivo, segun los archivos del mathematica.

Moreover, we have checked numerically that the following statement holds.

\begin{corollary}
Let $M$ be a torus of revolution immersed in $\sph^3$ with constant mean curvature $H$.
If $H\geq3/2$ or $H\leq-1$, then Ind($M$)$\geq8$.
\end{corollary}

\begin{remark}
By using \eqref{eq:c} and \eqref{eq:beta4}, it is also possible to express the eigenvalue
$1-\cosh^2(c)+1/\beta$ in terms of the mean curvature $H$ and the constant $\beta$
(which is equi\-valent to the first integral or energy of the generating curve $\ga$, see \eqref{eq:beta2}).
\end{remark}

\subsection{A degenerate case}
\label{subsec:degenerate}

We will now focus on the equality case $-1/\beta=1-\cosh^2(c)$.
We have verified that this equality may occur in our surfaces (for instance, when $H=1$ and $f(0)=0.4658$),
but in very few situations. In fact, it only holds when $H$ lies approximately in the interval $(-1,1.4)$,
and just for a unique value of $f(0)$ in each case.
Consequently, it can be considered as a degenerate possibility.
In these situations, we just have the following result:

%Let us suppose that the eigenvalues $-1/\beta$ and $1-\cosh^2(c)$ coincide in a CMC torus of revolution $M$
%immersed in $\sph^3$. Hence, it is clear that $-\cosh^2+1/\beta=-1$ and $1-\cosh^2+1/\beta=0$.
%%Moreover, as $u$, $\overline{u}$, $f_2$ are three eigenfunctions depending only on $t$ with the same eigenvalue,
%%for the second order differential equation given by $L_t$, it follows necessarily that they cannot be independent.
%%Then we have the following result.
%We easily have the following result.

\begin{theorem}
\label{th:degenerate}
Let $M$ be a CMC torus of revolution immersed in $\sph^3$.
If $1-\cosh^2(c)=-1/\beta$, then Ind($M$)$\geq5$.
\end{theorem}

\begin{proof}
Observe that, in this case, $u$, $\overline{u}$, $f_2$ are three eigenfunctions
depending only on $t$, with the same eigenvalue $-1/\beta$ for the second order differential equation given by $L_t$.
Then necessarily they cannot be independent, and so, taking into account Lemma~\ref{le:uu},
we can only assure that $-1/\beta$ is an eigenvalue of multiplicity two, at least.
Since $1-\cosh^2(c)+1/\beta=0$, from Theorem~\ref{th:six} and Lemma~\ref{le:otronegativo} we conclude that Ind($M$)$\geq5$.
\end{proof}

\begin{remark}
Note that the previous Theorem~\ref{th:degenerate} gives the same lower bound for the Morse index as in \cite{rs}.
\end{remark}

\begin{remark}
In~\cite{urbano}, it is proven that the Clifford torus has Morse index equal to five. For this torus we have the degenerate situation described in Subsection~\ref{subsec:degenerate}.
\end{remark}

%XXXX Seguir viendo si $u_i$ y $\sinh(w(t))$ son siempre independientes, o si en alguna ocasion hay dependencia. XXXX

Finally, we summarize the main results we have obtained
(Theorems~\ref{th:distintos} and \ref{th:degenerate}) in the following Theorem.

\begin{theorem}
\label{th:main}
Let $M$ be a CMC torus of revolution immersed in $\sph^3$.
With the previous notation, assume that $-1/\beta\neq1-\cosh^2(c)$.
Then,
\begin{itemize}
\item[-] if $\ $ $1-\cosh^2(c)+1/\beta<0$, then Ind($M$)$\geq 8$.
\item[-] If $\ $ $1-\cosh^2(c)+1/\beta>0$, then Ind($M$)$\geq 6$.
\end{itemize}
In particular, if $H\geq3/2$ or $H\leq-1$, then Ind($M$)$\geq 8$.

\noindent On the other hand, if $-1/\beta=1-\cosh^2(c)$, we have that Ind($M$)$\geq 5$.
\end{theorem}

With this result, taking into account that for each torus of revolution,
the constant values $c$ and $\beta$ can be computed by means of the expressions \eqref{eq:c} and \eqref{eq:beta4},
we obtain a lower bound for the Morse index in our surfaces.

\section{Some final comments}

\subsection{Study of equality of eigenvalues}
It is difficult to describe geometrically the tori of revolution in $\sph^3$
satisfying $-1/\beta=1-\cosh^2(c)$, which is the situation co\-rres\-ponding to the degenerate case from Subsection~\ref{subsec:degenerate}.
Anyway, this condition will impose some restrictions to the surface.

For the function $w(t)$, we have from \eqref{eq:betapositiva} that it depends on $b$,
or equivalently on the mean curvature $H$, thus involving the generating curve $\ga$ and its derivatives $\ga'$, $\ga''$ (observe the equations of the principal curvatures in \eqref{k1}).
In the degenerate case, $w(t)$ has a simpler expression: since $f_2$ must be a linear combination of $u$, $\overline{u}$
(see the proof of Th.~\ref{th:degenerate}), it follows that
\[
w(t)=\operatorname{arsinh}(-\rho_1\ga_1'(t)+\rho_2\ga_2'(t)),\ \rho_1,\rho_2\in\rr,
\]
that is, $w$ just depends essentially on $\ga'$.

Additionally, some restrictions appear between the derivatives of $\ga$.
Using again \eqref{eq:betapositiva} and the linear dependence of $f_2(t)=\sinh(w(t))$, $u(t)$ and $\overline{u}(t)$,
it is not difficult to find others analytic relations between $\ga_3$, $\ga_1'$ and $\ga_2'$.
Unfortunately, the geometrical meaning of them is not clear at all.

%By applying Courant's Nodal Domain Theorem~\cite{chavel}, one can estimates the position
%of a given eigenvalue in the associated sequence of eigenvalues. This fact can give some interesting
%information on the Morse index of a surface.

\subsection{An application of the Courant's Nodal Domain Theorem}
A further estimate of the Morse index of a surface can be obtained by means of
Courant's Nodal Domain Theorem.
This result states that for the ordered sequence $\{\lambda_n\}_n$ of eigenvalues,
the eigenfunction associated to an eigenvalue $\lambda_k$ has at most $k$ nodal domains
\cite[page 19]{chavel}.

In our case, fix for instance the eigenfunction $u(t)=-\gamma_1'(t)$, with associated eigenvalue $-1/\beta$.
The nodal domains for $u(t)$ will be determined by the
number of critical points (local minima or maxima) of the first coordinate of the
generating curve $\gamma(t)$. Hence, if $\gamma(t)$ have $k_0$ critical points
for the first coordinate, then $u(t)$ will have $k_0$ nodal domains, and so its
associated eigenvalue $-1/\beta=\lambda_k$ will satisfy $k\geq k_0$.
Consequently, since that eigenvalue is negative, we shall conclude that
the Morse index is, at least, $k_0$.
Observe that an analogous reasoning can be done by taking the eigenfunction $\overline{u}(t)=\gamma_2'(t)$.
Following this idea, a deeper analysis of the generating curves of particular tori of revolution in $\sph^3$
should improve the bounds for the Morse index in each case.

\subsection{The Morse index and instability}
As commented in the Introduction, the stability notion for CMC surfaces
is usually studied by considering volume-preserving variations.
This fact makes that the index form \eqref{eq:indexform} must be defined on the set of zero-mean functions (see \cite{bcs}).
As the eigenfunctions for computing the Morse index do not satisfy, in general, this restriction,
we recall the precise relation with stability (see \cite[Prop.~3.3]{bb} or \cite{rs}):
if a surface $M$ is stable, then $Ind$($M$)$\leq 1$.
An obvious consequence from our results is that any torus of revolution with constant mean curvature
immersed in $\sph^3$ is not stable (the only stable CMC surfaces in $\sph^3$ are, in fact, the round spheres~\cite{bcs}).

\bibliography{index}

\begin{thebibliography}{10}

\bibitem{alias}
Luis~J. Al{\'{\i}}as.
\newblock On the stability index of minimal and constant mean curvature
  hypersurfaces in spheres.
\newblock {\em Rev. Un. Mat. Argentina}, 47(2):39--61 (2007), 2006.

\bibitem{abp}
Luis~J. Al{\'{\i}}as, Aldir Brasil, Jr., and Oscar Perdomo.
\newblock On the stability index of hypersurfaces with constant mean curvature
  in spheres.
\newblock {\em Proc. Amer. Math. Soc.}, 135(11):3685--3693 (electronic), 2007.

\bibitem{bcs}
J.~Lucas Barbosa, Manfredo do~Carmo, and Jost Eschenburg.
\newblock Stability of hypersurfaces of constant mean curvature in {R}iemannian
  manifolds.
\newblock {\em Math. Z.}, 197(1):123--138, 1988.

\bibitem{bb}
Lucas Barbosa and Pierre B{\'e}rard.
\newblock Eigenvalue and ``twisted'' eigenvalue problems, applications to {CMC}
  surfaces.
\newblock {\em J. Math. Pures Appl. (9)}, 79(5):427--450, 2000.

\bibitem{chavel}
Isaac Chavel.
\newblock {\em Eigenvalues in {R}iemannian geometry}, volume 115 of {\em Pure
  and Applied Mathematics}.
\newblock Academic Press Inc., Orlando, FL, 1984.
\newblock Including a chapter by Burton Randol, With an appendix by Jozef
  Dodziuk.

\bibitem{choe}
Jaigyoung Choe.
\newblock Index, vision number and stability of complete minimal surfaces.
\newblock {\em Arch. Rational Mech. Anal.}, 109(3):195--212, 1990.

\bibitem{peng}
M.~do~Carmo and C.~K. Peng.
\newblock Stable complete minimal surfaces in {${\bf R}\sp{3}$} are planes.
\newblock {\em Bull. Amer. Math. Soc. (N.S.)}, 1(6):903--906, 1979.

\bibitem{ek}
Norio Ejiri and Motoko Kotani.
\newblock Index and flat ends of minimal surfaces.
\newblock {\em Tokyo J. Math.}, 16(1):37--48, 1993.

\bibitem{elsoufi}
Ahmad El~Soufi.
\newblock Applications harmoniques, immersions minimales et transformations
  conformes de la sph\`ere.
\newblock {\em Compositio Math.}, 85(3):281--298, 1993.

\bibitem{fc}
D.~Fischer-Colbrie.
\newblock On complete minimal surfaces with finite {M}orse index in
  three-manifolds.
\newblock {\em Invent. Math.}, 82(1):121--132, 1985.

\bibitem{guada}
I.~Guadalupe, Aldir Brasil, Jr., and J.~A. Delgado.
\newblock A characterization of the {C}lifford torus.
\newblock {\em Rend. Circ. Mat. Palermo (2)}, 48(3):537--540, 1999.

\bibitem{hsiang}
Wu-Yi Hsiang.
\newblock On generalization of theorems of {A}. {D}. {A}lexandrov and {C}.
  {D}elaunay on hypersurfaces of constant mean curvature.
\newblock {\em Duke Math. J.}, 49(3):485--496, 1982.

\bibitem{lr}
Francisco~J. L{\'o}pez and Antonio Ros.
\newblock Complete minimal surfaces with index one and stable constant mean
  curvature surfaces.
\newblock {\em Comment. Math. Helv.}, 64(1):34--43, 1989.

\bibitem{montielros}
Sebasti{\'a}n Montiel and Antonio Ros.
\newblock Schr\"odinger operators associated to a holomorphic map.
\newblock In {\em Global differential geometry and global analysis (Berlin,
  1990)}, volume 1481 of {\em Lecture Notes in Math.}, pages 147--174.
  Springer, Berlin, 1991.

\bibitem{naya}
Shin Nayatani.
\newblock Morse index and {G}auss maps of complete minimal surfaces in
  {E}uclidean {$3$}-space.
\newblock {\em Comment. Math. Helv.}, 68(4):511--537, 1993.

\bibitem{perdomo}
Oscar Perdomo.
\newblock Low index minimal hypersurfaces of spheres.
\newblock {\em Asian J. Math.}, 5(4):741--749, 2001.

\bibitem{ritros}
Manuel Ritor{\'e} and Antonio Ros.
\newblock Stable constant mean curvature tori and the isoperimetric problem in
  three space forms.
\newblock {\em Comment. Math. Helv.}, 67(2):293--305, 1992.

\bibitem{wente}
Wayne Rossman.
\newblock The {M}orse index of {W}ente tori.
\newblock {\em Geom. Dedicata}, 86(1-3):129--151, 2001.

\bibitem{rossman-tori}
Wayne Rossman.
\newblock Lower bounds for {M}orse index of constant mean curvature tori.
\newblock {\em Bull. London Math. Soc.}, 34(5):599--609, 2002.

\bibitem{rs}
Wayne Rossman and Nahid Sultana.
\newblock Morse index of constant mean curvature tori of revolution in the
  3-sphere.
\newblock {\em Illinois J. Math.}, 51(4):1329--1340, 2007.

\bibitem{rs2}
Wayne Rossman and Nahid Sultana.
\newblock The spectra of {J}acobi operators for constant mean curvature tori of
  revolution in the 3-sphere.
\newblock {\em Tokyo J. Math.}, 31(1):161--174, 2008.

\bibitem{simons}
James Simons.
\newblock Minimal varieties in riemannian manifolds.
\newblock {\em Ann. of Math. (2)}, 88:62--105, 1968.

\bibitem{tysk}
Johan Tysk.
\newblock Eigenvalue estimates with applications to minimal surfaces.
\newblock {\em Pacific J. Math.}, 128(2):361--366, 1987.

\bibitem{urbano}
Francisco Urbano.
\newblock Minimal surfaces with low index in the three-dimensional sphere.
\newblock {\em Proc. Amer. Math. Soc.}, 108(4):989--992, 1990.

\end{thebibliography}

\end{document}